\documentclass[11pt]{amsart}
\title[Cohomologically symplectic solvmanifolds]
{ Cohomologically symplectic solvmanifolds are symplectic }

\author{Hisashi Kasuya}
\usepackage{amssymb} 
\usepackage{amsmath} 
\usepackage{amscd}
\usepackage{amstext} 
\usepackage{amsfonts}

\usepackage{eucal}

\address[H.kasuya]{Graduate school of mathematical science university of tokyo japan }
\curraddr{}
\email{khsc@ms.u-tokyo.ac.jp}

\thanks{}

\newcommand{\C}{\mathbb{C}}
\newcommand{\R}{\mathbb{R}} 

\newcommand{\Ra}{\mathbb{Q}}
\newcommand{\Z}{\mathbb{Z}}
\newcommand{\g}{\frak{g}}

\theoremstyle{plain}
\newtheorem{theorem}{Theorem}[section] 
\theoremstyle{remark}
\newtheorem{remark}{Remark}
\theoremstyle{lemma}

\theoremstyle{definition}
\newtheorem{definition}[theorem]{Definition}
\theoremstyle{proposition}
\newtheorem{proposition}[theorem]{Proposition}
\theoremstyle{corollary}

\begin{document} 
\begin{abstract} 
We consider  aspherical manifolds with torsion-free virtually polycyclic fundamental groups, constructed by Baues.
We prove that if those manifolds are cohomologically symplectic then they are symplectic.
As a corollary we show that cohomologically symplectic solvmanifolds are symplectic.

\end{abstract}
\maketitle

\section{\bf Introduction}

A $2n$-dimensional compact manifold $M$ is called cohomologically symplectic (c-symplectic)  if we have $\omega\in H^{2}(M,\R)$ such that $\omega^{n}\not=0$.
A compact symplectic manifold is c-symplectic but the converse is not true in general.
For example $\C P^{2}\#\C P^{2}$ is c-symplectic but not symplectic. 
But for some class of manifolds these two conditions are equivalent.
For examples, nilmanifolds i.e. compact homogeneous spaces of  nilpotent simply connected  Lie group. In \cite{Nom}, for a nilpotent  simply connected Lie group $G$ with a cocompact discrete subgroup $\Gamma$ (such subgroup is called a lattice), Nomizu showed that the De Rham cohomology $H^{\ast}(G/\Gamma, \R)$ of $G/\Gamma$ is isomorphic to the cohomology $H^{\ast}(\g)$ of the Lie algebra of $G$.
By the application of Nomizu's theorem, if $G/\Gamma$ is c-symplectic then $G/\Gamma$ is symplectic (see \cite[p.191]{Fe}).
Every nilmanifold can be represented by such $G/\Gamma$ (see \cite{Mal}).

Consider Solvmanifolds i.e. compact homogeneous spaces of  solvable simply connected  Lie groups.
Let $G$ be a solvable simply connected Lie group with a lattice $\Gamma$. 
We assume that for any $g\in G$ the all eigenvalues of the adjoint operator ${\rm Ad}_{g}$ are real.
With this assumption, in \cite{Hatt} Hattori extended Nomizu's theorem.
By Hattori's theorem, for such case, without difficulty, we can similarly show that if $G/\Gamma$ is c-symplectic, then $G/\Gamma$ is symplectic.
But  the isomorphism $H^{\ast}(G/\Gamma, \R)\cong H^{\ast}(\g)$ fails to hold  for general solvable Lie groups, and not all solvmanifolds can be represented by $G/\Gamma$.
Thus it is a considerable problem whether every c-symplectic solvmanifold is symplectic.

Let $\Gamma$ be a torsion-free virtually polycyclic group.
In \cite{B} Baues constructed the compact aspherical manifold $M_{\Gamma}$ with $\pi_{1}(M_{\Gamma})=\Gamma$.
Baues proved that every infra-solvmanifold (see \cite{B} for the definition) is diffeomorphic to $M_{\Gamma}$. In particular the class of such aspherical manifolds contains the class of solvmanifolds.  
We prove that if $M_{\Gamma}$ is  c-symplectic then $M_{\Gamma} $ is symplectic.
In other words, for a torsion-free virtually polycyclic group $\Gamma$ with $2n={\rm rank}\, \Gamma$,  if there exists $\omega\in H^{2}(\Gamma,\R)$ such that $\omega^{n}\not=0  $ then we have a  symplectic aspherical manifold with the fundamental group $\Gamma$.

\section{\bf Notation and conventions}

Let $k$ be a subfield of $\C$.
A group $\bf G$ is called a $k$-algebraic group if $\bf G$ is a Zariski-closed subgroup of $GL_{n}(\C)$ which is defined by polynomials with coefficients in $k$.
Let  ${\bf G}(k)$ denote the set  of  $k$-points of $\bf G$ 
and ${\bf U}({\bf G})$ the maximal Zariski-closed unipotent normal $k$-subgroup of $\bf G$ called the unipotent radical of $\bf G$.
Let $U_{n}(k)$ denote the $n\times n$ $k$-valued upper triangular unipotent matrix group.

\section{\bf Aspherical  manifolds with torsion-free virtually polycyclic fundamental groups}

\begin{definition}\label{a-d-1}
A group $\Gamma$ is polycyclic if it admits a sequence 
\[\Gamma=\Gamma_{0}\supset \Gamma_{1}\supset \cdot \cdot \cdot \supset \Gamma_{k}=\{ e \}\]
of subgroups such that each $\Gamma_{i}$ is normal in $\Gamma_{i-1}$ and $\Gamma_{i-1}/\Gamma_{i}$ is cyclic.
We denote ${\rm rank}\,\Gamma=\sum_{i=1}^{i=k}{\rm rank}\, \Gamma_{i-1}/\Gamma_{i}$.
\end{definition}

\begin{proposition}{\rm (\cite[Proposition 3.10]{R})}
 The fundamental group of a solvmanifold is  torsion-free polycyclic.
\end{proposition}

Let $k$ be a subfield of $\mathbb{C}$.
Let $\Gamma$ be a torsion-free virtually polycyclic group.
For a finite index polycyclic subgroup $\Delta\subset \Gamma$, we denote ${\rm rank}\, \Gamma ={\rm  rank}\, \Delta$.

\begin{definition}\label{a-d-2}
We call a $k$-algebraic group  ${\bf H}_{\Gamma}$ a $k$-algebraic hull of $\Gamma$ if there exists an injective group homomorphism $\psi:\Gamma\to {\bf H}_{\Gamma}(k)$
and ${\bf H}_{\Gamma}$ satisfies the following conditions:
\\
(1)  \ $\psi (\Gamma)$ is Zariski-dense in $\bf{H}_{\Gamma}$.\\
(2) \   $Z_{{\bf H}_{\Gamma}}({\bf U}({\bf H}_{\Gamma}))\subset {\bf U}({\bf H}_{\Gamma})$ where  $Z_{{\bf H}_{\Gamma}}({\bf U}({\bf H}_{\Gamma}))$ is the centralizer of ${\bf U}({\bf H}_{\Gamma})$. \\
(3) \ $\dim {\bf U}({\bf H}_{\Gamma})$=${\rm rank}\,\Gamma$.   
\end{definition}
\begin{theorem}\label{a-t-1}{\rm (\cite[Theorem A.1]{B})}
There exists a $k$-algebraic hull of $\Gamma$ and a k-algebraic hull of $\Gamma$ is unique up to $k$-algebraic group isomorphism.
\end{theorem}

 Let $\Gamma$ be a torsion-free virtually polycyclic group and
$\bf{H_{\Gamma}}$  the $\Ra$-algebraic hull of $\Gamma$.
Denote $H_{\Gamma}=\bf{H_{\Gamma}}(\R)$. 
Let $U_{\Gamma}$ be the unipotent radical of $H_{\Gamma}$ and $T$  a maximal reductive subgroup.
Then $H_{\Gamma}$ decomposes as a semi-direct product $H_{\Gamma}=T\ltimes U_{\Gamma}$.
Let $\frak{u}$ be the Lie algebra of $U_{\Gamma}$. Since the exponential map ${\exp}:{\frak u} \longrightarrow U_{\Gamma}$ is a diffeomorphism, $U_{\Gamma}$ is diffeomorphic to $\R^n$ such that $n={\rm rank}\,\Gamma$.
For the semi-direct product $H_{\Gamma}=T\ltimes U_{\Gamma}$, we denote $\phi:T\to {\rm Aut}(U_{\Gamma})$ the action of $T$ on $U_{\Gamma}$.
Then we have the homomorphism $\alpha :H_{\Gamma}\longrightarrow {\rm Aut}(U_{\Gamma})\ltimes U_{\Gamma}$ such that $\alpha(t,u)=(\phi(t), u)$ for $(t,u)\in T\ltimes U_{\Gamma}$.
By the property (2) in Definition \ref{a-d-2}, $\phi$ is injective and hence $\alpha$ is injective.

In \cite{B} Baues constructed a compact  aspherical manifold $M_{\Gamma}=\alpha(\Gamma)\backslash U_{\Gamma}$ with $\pi_{1}(M_{\Gamma})=\Gamma$.
We call $M_{\Gamma}$ a standard $\Gamma$-manifold.

\begin{theorem}{\rm (\cite[Theorem 1.2, 1.4]{B})}\label{a-t-3}
A standard $\Gamma$-manifold is unique up to diffeomorphism.
A solvmanifold with the fundamental group $\Gamma$ is diffeomorphic to the standard $\Gamma$-manifold $M_{\Gamma}$.
\end{theorem}
Let $A^{\ast}(M_{\Gamma})$ be the de Rham complex of $M_{\Gamma}$.
Then $A^{\ast}(M_{\Gamma}) $ is  the set of   the $\Gamma$-invariant differential forms ${A^{\ast}(U_{\Gamma})}^{\Gamma}$ on $U_{\Gamma}$. 
Let $(\bigwedge \frak{u} ^{\ast})^{T}$ be the left-invariant forms on $U_{\Gamma}$ which are fixed by $T$.
Since $\Gamma\subset H_{\Gamma}=U_{\Gamma}\cdot T$, we have the inclusion
\[(\bigwedge {\frak u} ^{\ast})^{T} ={A^{\ast}(U_{\Gamma})}^{H_{\Gamma}} \subset {A^{\ast}(U_{\Gamma})}^{\Gamma}= A^{\ast}(M_{\Gamma}).\]
\begin{theorem}{\rm (\cite[Theorem 1.8]{B})}\label{a-t-4}
This inclusion induces an  isomorphism on cohomology.
\end{theorem}

By the application of the above facts, we prove the main theorem of this paper.

\begin{theorem}\label{symex}
Suppose $M_{\Gamma} $ is c-symplectic. 
Then $M_{\Gamma}$ 
admits a symplectic structure.
In particular cohomologically symplectic solvmanifolds are symplectic.
\end{theorem}
\begin{proof}
Since we have the isomorphism $H^{\ast}(M_{\Gamma},\R)\cong H^{\ast}((\bigwedge {\frak u}^{\ast})^{T})$, we have $\omega\in (\bigwedge^{2} {\frak u}^{\ast})^{T}$ such that $0\not=[\omega]^{n}\in H^{2n}((\bigwedge 
{\frak u}^{\ast})^{T})$.
This gives $0\not= \omega^{n}\in  (\bigwedge {\frak u}^{\ast})^{T}$ and 
hence $0\not= \omega^{n}\in  \bigwedge {\frak u}^{\ast}$.
Since  $\omega^{n}$ is a non-zero invariant $2n$-form on $U_{\Gamma}$, we have
$(\omega^{n})_{p}\not= 0$ for any $p\in U_{\Gamma}$.
Hence by the inclusion $(\bigwedge {\frak u}^{\ast})^{T}\subset A^{\ast}(U_{\Gamma})^{T}=A^{\ast}(M_{\Gamma})$, 
we have $(\omega^{n})_{\Gamma p}\not= 0$ for any $\Gamma p\in \Gamma \backslash U_{\Gamma}=M_{\Gamma} $.
This implies that $\omega$ is a symplectic form on $M_{\Gamma}$.
Hence we have the theorem. 
\end{proof}
\section{remarks}
Let $G=\R\ltimes_{\phi}U_{3}(\C)$ such that
\[\phi(t)\cdot \left(\begin{array}{ccc}
1&x&z\\
0&1&y\\
0&0&1\\
\end{array}
\right)=
\left(\begin{array}{ccc}
1&e^{i\pi t}\cdot x&z\\
0&1&e^{-i\pi t}\cdot y\\
0&0&1\\
\end{array}
\right),\]
and $D=\Z\ltimes_{\phi} D^{\prime}$ with 
\[D^{\prime}=\left\{ \left(
\begin{array}{ccc}
1&x_{1}+ix_{2}&z_{1}+iz_{2}\\
0&1&y_{1}+iy_{2}\\
0&0&1\\
\end{array}
\right)  : x_{1}, y_{2}, z_{2} \in \Z,\, x_{2},y_{1},z_{1}\in \R \right\}.\]
Then $D$ is not discrete and $G/D$ is compact.
We have $D/D_{0}\cong \Z\ltimes_{\varphi }U_{3}(\Z)$ such that
\[\varphi(t)\cdot\left(\begin{array}{ccc}
1&x&z\\
0&1&y\\
0&0&1\\
\end{array}
\right)=\left(\begin{array}{ccc}
1&(-1)^{t}x&z\\
0&1&(-1)^{-t}y\\
0&0&1\\
\end{array}
\right),\]
where $D_{0}$ is the identity component of $D$.
Denote $\Gamma=D/D_{0}$.
We have the algebraic hull $H_{\Gamma}=\{\pm 1\}\ltimes_{\psi }(U_{3}(\R)\times \R)$ such that
\[\psi(-1)\cdot \left(\left(\begin{array}{ccc}
1&x&z\\
0&1&y\\
0&0&1\\
\end{array}
\right), t\right)=\left(\left(\begin{array}{ccc}
1&-x&z\\
0&1&-y\\
0&0&1\\
\end{array}
\right), t\right).\]
The dual of the Lie algebra ${\frak u}$ of $U_{3}(\R)\times \R$
is given by ${\frak u}^{\ast}=\langle \alpha ,\beta ,\gamma ,\delta \rangle$ such that the differential is given by
\[ d \alpha =d \beta =d\delta =0,\]
\[ d \gamma =-\alpha \wedge \beta,\]
and the action of $\{ \pm 1\}$ is given by 
\[(-1)\cdot \alpha =-\alpha ,\, (-1)\cdot \beta =-\beta , \]
\[(-1)\cdot\gamma =\gamma ,\, (-1)\cdot\delta =\delta .\] 
Then we have a diffeomorphism $M_{\Gamma}\cong G/D$ and an isomorphism $H^{\ast}(M_{\Gamma}, \R)$ $\cong H^{\ast}((\bigwedge {\frak u}^{\ast})^{\{\pm 1\}})$.
By simple computations,
 $H^{2}( (\bigwedge {\frak u}^{\ast})^{\{\pm 1\}})=0$
 and hence the solvmanifold $G/D$ is not symplectic.
\begin{remark}
The proof of the Theorem \ref{symex} contains a proof of the following proposition.
\begin{proposition}\label{nos}
If $M_{\Gamma}$ admits a symplectic structure, then $U_{\Gamma}$ has an 
invariant symplectic form.
\end{proposition} 
Otherwise for the above example, $U_{\Gamma}=U_{3}(\R)\times \R$ has an invariant symplectic form but $M_{\Gamma}$ is not symplectic.
Thus the converse of this proposition is not true.
If $\Gamma$ is nilpotent, then $T$ is trivial and any invariant symplectic form on $U_{\Gamma}$ induces the symplectic form on $M_{\Gamma}$.
Hence for nilmanifolds the converse of Proposition \ref{nos} is true.
\end{remark}
\begin{remark}
$\Gamma$ is a finite extension of a lattice of $U_{\Gamma}=U_{3}(\R)\times \R$.
Hence $M_{\Gamma}$ is finitely covered by a Kodaira-Thurston manifold (see \cite{Th}, \cite[p.192]{Fe}).
$M_{\Gamma}$ is an example of a non-symplectic manifold finitely covered by a symplectic manifold.
\end{remark}

Let $H=G\times \R$. 
Then the dual of the Lie algebra ${\frak h}$ of $H$ is given by
${\frak h}^{\ast}=\langle \sigma ,\tau, \zeta _{1}, \zeta _{2}, \eta _{1},\eta _{2},\theta _{1},\theta _{2}\rangle$ such that the differential is given by 
\[d\sigma = d \tau =0,\]
\[d \zeta _{1}= \tau \wedge \zeta _{2},\, d\zeta _{2} =-\tau \wedge \zeta _{1},\]
\[d\eta _{1}=\tau \wedge \eta _{2},\, d \eta _{2}=-\tau \wedge \eta _{1},\]
\[d\theta _{1}= -\zeta _{1}\wedge \eta_{1} +\zeta _{2}\wedge \eta _{2},\,
d\theta _{2}= -  \zeta _{1}\wedge \eta_{2} -\zeta _{2}\wedge \eta _{1}.\]
By simple computations, any closed invariant $2$-form $\omega\in \bigwedge^{2}{\frak h}^{\ast}$ satisfies $\omega^{4}=0$.
Hence $H$ has no invariant symplectic form.
Otherwise we have a lattice $\Delta=2\Z\ltimes U_{3}(\Z+i\Z)\times \Z$ which is also a lattice of $\R^{2}\times U_{3}(\C)$.
Thus $H/\Delta$ is diffeomorphic to a direct product of a $2$-dimensional torus and an Iwasawa manifold (see \cite{FG}).
Since an Iwasawa manifold is symplectic (see \cite{FG}), $H/\Delta$ is also symplectic.
By this example we can say:
\begin{remark}
For a simply connected nilpotent Lie group $G$ with lattice $\Gamma$,
if the nilmanifold $G/\Gamma$ is symplectic then $G$ has an invariant symplectic form.
But suppose $G$ is solvable we have an example of a symplectic solvmanifold $G/\Gamma$ such that $G$ has no invariant symplectic form. 
\end{remark}
{\bf  Acknowledgements.} 

The author would like to express his gratitude to   Toshitake Kohno for helpful suggestions and stimulating discussions.
This research is supported by JSPS Research Fellowships for Young Scientists.


\begin{thebibliography}{20}


\bibitem{B}
O. Baues, Infra-solvmanifolds and rigidity of subgroups in solvable linear algebraic groups. Topology {\bf 43} (2004), no. 4, 903--924.
\bibitem{Bor}
A. Borel, Linear algebraic groups 2nd enl. ed Springer-verlag (1991).
\bibitem{Fe}
Y. F$\acute {\rm e}$lix, J. Oprea and D. Tanr$\acute {\rm e}$, Algebraic Models in Geometry, Oxford Graduate Texts in
Mathematics 17, Oxford University Press 2008.
\bibitem{FG} M. Fernandez, A. Gray, The Iwasawa manifold. Differential 
geometry, Peniscola 1985, 157--159, Lecture Notes in Math., 1209, 
Springer, Berlin, 1986.

\bibitem{Hatt} A. Hattori, Spectral sequence in the de Rham cohomology of fibre bundles. J. Fac. Sci. Univ. Tokyo Sect. I {\bf 8} 1960 289--331 (1960). 
\bibitem{Mal}
A. Malcev, On a class of homogeneous spaces. (Russian) Izvestiya Akad. Nauk. SSSR. Ser. Mat. {\bf 13}, (1949). 9--32.
\bibitem{Nom}
K. Nomizu, On the cohomology of compact homogeneous spaces of nilpotent Lie groups. 
Ann. of Math. (2) {\bf 59}, (1954). 531--538.
 \bibitem{R}
 M. S. Raghunathan, Discrete subgroups of Lie Groups, Springer-Verlag, New York, 1972. 
 
 \bibitem{Th} W. P. Thurston, Some simple examples of symplectic 
manifolds. 
Proc. Amer. Math. Soc. {\bf 55} (1976), no. 2, 467--468. 
\end{thebibliography}
\end{document}